\documentclass[final]{siamonline1116}

\usepackage{slashbox}
\usepackage[utf8]{inputenc}
\usepackage{bm}
\usepackage{graphicx}
\usepackage{amsmath}
\usepackage[parfill]{parskip}
\usepackage{xcolor}
\usepackage[english]{babel}
\usepackage{amsfonts}
\usepackage{amssymb}
\usepackage{amsbsy}
\usepackage{hyperref}
\usepackage{enumitem}
\usepackage{epstopdf}
\usepackage{subcaption}
\usepackage{cleveref}

\ifpdf
  \DeclareGraphicsExtensions{.eps,.pdf,.png,.jpg}
\else
  \DeclareGraphicsExtensions{.eps}
\fi

\numberwithin{theorem}{section}

\newsiamremark{remark}{Remark}
\newsiamremark{Example}{Example}
\newsiamthm{assumption}{Assumption}
\newsiamthm{Conjecture}{Conjecture}

\newcommand{\R}{\mathbb{R}}

\newcommand{\E}{\mathcal{E}}
\newcommand{\Lt}{L^2(\Omega)}

\renewcommand{\u}{\vec{u}}

\newcommand{\f}{\vec{f}}
\newcommand{\y}{\vec{y}}
\newcommand{\e}{\vec{\eta}}
\newcommand{\w}{\vec{\omega}}
\newcommand{\hw}{\hat{\vec{\omega}}}

\renewcommand{\div}{\operatorname{div}}
\newcommand{\curl}{\operatorname{curl}}

\newcommand{\TheTitle}{Exponentially convergent data assimilation algorithm for Navier-Stokes equations}
\newcommand{\ShortTitle}{Data assimilation for NSE}
\newcommand{\TheAuthors}{Jason Frank, Tigran Tchrakian, Sergiy Zhuk}

\ifpdf
\hypersetup{
  pdftitle={\TheTitle},
  pdfauthor={\TheAuthors}
}
\fi

\headers{\ShortTitle}{\TheAuthors}

\title{{\TheTitle}\thanks{Published in Proceedings of American Control Conference, 2017, DOI: 10.23919/ACC.2017.7963448
}}

\author{Jason Frank\thanks{Utrecht University, The Netherlands, \email{j.e.frank@uu.nl}
}
\and Tigran T. Tchrakian\thanks{IBM Research, Dublin, Ireland, \email{tigran@ie.ibm.com}
}
\and Sergiy Zhuk\thanks{IBM Research, Dublin, Ireland, \email{sergiy.zhuk@ie.ibm.com}}
}

\usepackage{amsopn}

\begin{document}
\maketitle
\begin{abstract}
The paper presents a new state estimation algorithm for a bilinear equation representing the Fourier-Galerkin (FG) approximation of the Navier-Stokes (NS) equations on a torus in $\mathbb{R}^2$. This state equation is subject to uncertain but bounded noise in the input (Kolmogorov forcing) and initial conditions, and its output is incomplete and contains bounded noise. The algorithm designs a time-dependent gain such that the estimation error converges to zero exponentially. The sufficient condition for the existence of the gain are formulated in the form of algebraic Riccati equations. To demonstrate the results we apply the proposed algorithm to the reconstruction a chaotic fluid flow from incomplete and noisy data.
\end{abstract}

\section{Introduction}
\label{sec:introduction}
Data Assimilation (DA) improves the accuracy of forecasts, provided by physical models, and evaluates their reliability by optimally combining \emph{a priori} knowledge encoded in equations of mathematical physics with \emph{a posteriori} information in the form of sensor data (state estimation). Mathematically, DA relies upon optimal control methods or applied probability. In the probabilistic framework, the state estimation problem is solved by means of the stochastic filtering algorithm. Namely, evolution of the conditional density of the states of a Markov diffusion process is described by a stochastic PDE, the so-called Kushner-Stratonovich (KS) equation~\cite{GihmanSkorokhod1997}.
%An alternative approach, so called 4DVAR, is based on the fact that locally\footnote{If $u(s)$ is the state of the diffusion process and $u(s)=x$ then $u(s+\delta s)-x = a(x,s)\delta s+b+o((\delta s)^2)$ where $a$ is the drift and $b$ is a Gaussian random variable with zero mean and correlation function $b(x,s)\delta s$, $b$ is the diffusion coefficient.} the state of the Markov process is a Gaussian random variable~\cite{GihmanSkorokhod1997}. According to this, 4DVAR suggests to minimize a log-density, an error cost function which is obtained by taking a logarithm of a Gaussian density representing the model and observation errors along the path of the diffusion process~\cite{StuartDA2015}. In the case of linear systems, both approaches provide the same state estimate. %Indeed, KS equation transforms into the classical Kalman-Bucy filter and 4DVAR transforms into a classical linear quadratic control problem such that its solution in the feed-back form is given by Riccati equation which, in fact, describes the evolution of the state error covariance matrix of the Kalman-Bucy filter. In the general case of nonlinear dynamics, 4DVAR may find a local optimizer of the log-density which is different from the estimate provided by a filter.

Deterministic state estimators, including the algorithm presented in this paper, assume that errors have bounded energy and belong to a given bounding set. The state estimate is then defined as a minimax center of the reachability set, a set of all states of the physical model which are reachable from the given set of initial conditions and are compatible with observations. Dynamics of the minimax center is described by a minimax filter. The latter may be constructed by using dynamic programming, i.e., the set $V\le 1$, where $V$ is the so-called value function $V$ solving a Hamilton-Jacobi-Bellman (HJB) equation~\cite{Bardi1997}, coincides with the reachability set~\cite{Baras1995}. Statistically, the uncertainty description in the form of a bounding set represents the case of uniformly distributed bounded errors in contrast to stochastic filtering, where all the errors are usually assumed to be in the form of ``white noise''. However, in many cases (e.g. linear dynamics and ellipsoidal uncertainty description) $\exp\{-V\}$, coincides with the solution of the KS equation. Moreover, the solution of the HJB equation, $V$, may be represented as a non-negative quadratic form, and by computing the exponential of $-V$, one obtains the conditional density of the linear Markov process which also solves KS equation. In fact, the inverse of the Riccati operator coinsides with the state error covariance matrix of the Kalman-Bucy filter. In the nonlinear case the link between deterministic state estimation and stochastic filtering becomes less obvious. %However, the notion of the viscosity solution of the HJB equation~\cite{EvansPDE2010} reveals the following connection: the viscosity solution of the HJB equation may be obtained as a limit of a Cole-Hopf transform of the solution of forward Kolmogorov equation with a specific source term, provided the diffusion coefficient is sent to zero~\cite{EvansPDE2010}. The aforementioned source term is, in fact, similar to the innovation term of the Kushner-Stratonovich equation, namely it represents a square of the Euclidean distance between observations and states of the physical model.

For generic nonlinear models both minimax and stochastic filters are infinite-dimensional: to get an optimal estimate one needs to solve a PDE (either the KS or HJB equation) in $\R^n$. Hence, if the state space of the original physical model is high-dimensional (e.g., a model representing the FG approximation of Navier-Stokes equations in 2D) then both filters become computationally intractable due to the ``curse of dimensionality''. Tractable approximations of optimal filters are briefly reviewed below. An overview of modern data assimilation methods is given in~\cite{ReichDA2015,StuartDA2015}.

The most popular approximations of optimal filters include the Extended Kalman Filter (ExKF), the Ensemble Kalman Filter (EnKF) and Luenberger/high gain observers. ExKF is based on the following idea: given an accurate estimate of the state at time instant $t$, one ``linearizes the dynamics'' around that estimate and applies Kalman filtering for the resulting linear system to obtain an estimate for the next time step. This procedure is then repeated. The major drawback of ExKF is that it may diverge for nonlinear equations with positive Lyapunov exponents. A computational bottleneck associated with ExKF is the requirement to recompute the state error covariance matrix, the gain. The EnKF overcomes this issue by generating an ensemble of trajectories and by computing the ensemble variance to approximate the gain. The latter is then used to compute a state estimate in the same way as in the Kalman filter, i.e., it uses a standard formula that allows one to obtain the distribution of a Gaussian random variable $\eta$ given a realizations of $\xi$, provided $\eta$ and $\xi$ have joint Gaussian distribution. A so-called asymptotic observers or Luenberger observers do not require optimal gain matrices (e.g. Riccati matrices). Instead, the gain is chosen so that the dynamics of the estimation error is described by an asymptotically stable linear (Luenberger observers) or non-linear (high-gain observers) ODE, so that the estimation error associated with the corresponding state estimator asymptotically approaches zero.

In this paper we design an exponentially convergent state estimator for a so called vorticity equation, the vorticity-streamfunction formulation of the Navier-Stokes (NS) equations in two spatial dimensions~\cite{MajdaBertozzi2002}. The vorticity equation is subject to uncertain but bounded noise in the input (Kolmogorov forcing) and initial conditions, and its output is incomplete and contains bounded noise. Assuming periodic boundary conditions, we apply Fourier-Galerkin (FG) approximation, i.e., we project the vorticity equation onto a $2N+1$-dimensional subspace generated by $\{e^{ikx}e^{isy}\}_{|k|,|s|\le \frac N2}$ and obtain an ODE for the projection coefficients, a FG model (see Section~\ref{sec:euler}). Note that Fourier-Galerkin approximation possesses a spectral convergence rate provided the solution of the vorticity equation is smooth~\cite{BardosTadmor2015}. %Although the main results of this work may be derived for the case of bounded domains with non-penetration boundary conditions or unbounded domains ($\R^2$), we focus on smooth periodic flows to simplify the convergence analysis.

Design of our state estimator relies upon the following ``key observation'': the bilinear convective operator of the vorticity equation is skew-symmetric. The same holds true for the bilinear term in the FG model which represents the FG discretization of the convective operator in the FG model. This fact allows us to show that the dynamics of the Euclidian norm of the estimation error is, in fact, independent of the bilinear convective term. This, in turn, is used to construct a time-dependent gain for the state estimator such that the estimation error converges to zero asymptotically. In the general case of noisy outputs the gain is constructed as a solution of a non-stationary algebraic Riccati inequality which reduces to a Linear Matrix Inequality (LMI) provided the output is exact. As a result, it is sufficient to solve an algebraic matrix Riccati inequality to get the exponential convergence for the corresponding state estimator. For the LMI case we use the least-squares solution of the corresponding algebraic Lyapunov equation (in continuous time). The latter allows us to introduce sufficient conditions for the detectability of the FG model: the real spectrum of the residual of the algebraic Lyapunov equation (in continuous time) must belong to $(-\infty,0)$ (see Section~\ref{sec:discretization}).  The numerical study demonstrates that in some cases the estimation error converges to zero even though the proposed detectability conditions are not fulfilled (see Section~\ref{sec:num-example}).

To the best of our knowledge, this result is new and easily generalizes to generic bilinear equations with skew-symmetric nonlinearity (e.g. Lorenz 96 model, Burgers equations). Recently, a few fully justified estimators for bilinear equations have appeared in the literature: an ellipsoidal state estimator~\cite{FillipovaQS2008}, a 3DVAR algorithm for the incompressible Navier-Stokes equations in 2D~\cite{Stuart3DVAR2013}, and the minimax filter for the Euler equations in 2D~\cite{ZhukTTCDC15}. The first algorithm is based on ellipsoidal calculus~\cite{Kurzhanski1997} and approximates the reachability set by an ellipsoid. This approach is reported to work well in small dimensions~\cite{FillipovaQS2008} and tends to overestimate the reachability set when the dimension of the ODE increases. The 3DVAR algorithm has been justified for Navier-Stokes equations on a torus in 2D and represents the most basic form of the filter which does not account for the model error and does not update the state error covariance matrix. The latter is very attractive from the computational stand-point as the stationary error covariance matrix is not expensive to propagate. However, the quality of the state estimate provided by 3DVAR strongly depends upon the choice of the covariance matrix. Finally, the minimax filter of~\cite{ZhukTTCDC15} uses a different gain design which is more expensive computationally and does not guarantee the exponential convergence. A similar approach has been used to design data assimilation algorithms for bilinear traffic flow models~\cite{ZhukITS14}. Adaptive parameter estimators for hyperbolic equations were considered for instance in~\cite{Demetriou1998}.

\section{Mathematical preliminaries}
\label{sec:euler}

\paragraph{Notation} Let $\Omega:=(0,L_x)\times(0,L_y)$ denote a rectangle with boundary $\partial\Omega$, and $\vec{n}(x,y)$ is a unit vector pointing outside $\Omega$ such that $\vec{n}(x,y)\perp\partial\Omega$, $\Omega_T:=\Omega\times(0,T)$. $C^s(\Omega)$ denotes a space of continuously differentiable functions on $\Omega$ (up to order $s$), $\Lt$ is the space of square-integrable functions on $\Omega$ with inner product $(f,g)_{\Lt}:=\frac1{L_xL_y}\int_\Omega f(x,y)\overline{g}(x,y)dxdy$, $\overline{g}$ is the complex conjugate of $g$, $x\cdot y$ is the canonical inner product of vectors $x,y$, $\Psi^\top$ is the transposed matrix, $\Psi^\star$ is the complex conjugate of $\Psi^\top$, $(x,y)_{C^n}:=x\cdot{\bar y}$ for complex vectors $x,y$ with $n$-components. $H^1(\Omega)$ is a Sobolev space of $L^2(\Omega)$-functions with weak first derivatives of $L^2(\Omega)$-class. $L^2(t_0,t_1,H):=\{f:f(t)\in H\text{ and } \int_{t_0}^{t_1} \| f(t)\|^2_H dt <+\infty\}$. We write $u=v$ a.e. on $\Omega$ if $u(x)=v(x)$ for almost all $x\in\Omega$.

Set $\div(\u)=\partial_{x_1}u_1+\partial_{x_2}u_2$, $\curl(\u)=\partial_{x_1}u_2-\partial_{x_2}u_1$, $\nabla u = (\partial_{x_1}u,\partial_{x_2}u)^\top$, $\nabla^\perp u = (-\partial_{x_2}u,\partial_{x_1}u)^\top$.
Given a vector-function $\u$, define \(
b(\u,w,v):=(\u\cdot\nabla w,\overline{v})_{\Lt}
\) and set $a(\psi,\phi) = (\nabla\psi,\nabla\phi)_{\Lt}$. Define $\phi_c(x):=e^{\frac{2\pi i cx}{L_x}}$, $\phi_d(y):=e^{\frac{2\pi i dy}{L_y}}$ and set $\phi_{cd}(x,y):=\phi_c(x)\phi_d(y)$. $e_j$ denotes the $j$th canonical basis vector in $\R^n$. Finally, let $\lambda_{cd}:=\frac{4\pi^2 c^2}{L_x^2} + \frac{4\pi^2 d^2}{L_y^2}$, provided $c^2+d^2>0$ and $\lambda_{cd}=0$ for the case $c^2+d^2=0$.
%$H^2:=H\times H$ denotes the cartesian product of $H$ with itself.

\paragraph{Fourier-Galerkin model for Navier-Stokes equations} Assume that $\omega$ verifies the weak vorticity-streamfunction formulation of the Navier-Stokes equations:
\begin{equation}
 \label{eq:vorticity}
 \begin{split}
 \dfrac{d}{dt}&(\omega,\phi)_{\Lt} + b(\u+\nabla^\perp\psi,\omega,\phi)+\nu a(\omega,\phi)=(\mathcal{D}f,\phi)_{\Lt}\,,(x,t)\in\Omega_T\,,\\
 &a(\psi,\phi) = (\omega,\phi)_{\Lt}\,,\omega(0)=\curl(\u_0)\,,
%&\partial_t \omega +\u\cdot\nabla\omega =\Delta\omega + \mathcal{D}f\,,\quad -\Delta \psi = \omega\,,\\
%&\u=\bar u+\nabla^\perp\psi\,,\omega(0)=\curl(\u_0)\,,\\
\end{split}
\end{equation}
where $\u = (\tilde u,\tilde v)^\top$ is a given vector representing the mean velocity field, $\nu>0$ is the diffusion coefficient, $\u_0\in C^2(\Omega)^2$ is the initial velocity and $f\in C^1(\Omega_T)$ has zero mean, $\int_\Omega f(x,t) dx =0$, $\mathcal D$ is a given bounded linear operator in $\Lt$.\\
Note that the weak formulation~\eqref{eq:vorticity} encapsulates various boundary conditions. Indeed, recall that according to Green's formula, one has: \begin{equation}
a(\omega,\phi) = -(\Delta\omega,\phi)_{\Lt} + (\nabla\omega\cdot\vec{n},\phi)_{L^2(\partial\Omega)}\,.
\end{equation} It then follows that $(\nabla\omega\cdot\vec{n},\phi)_{L^2(\partial\Omega)}=0$
%$\int_{\partial\Omega} (\nabla\omega\cdot\vec{n})\phi d\partial\Omega=0$
in the following cases:
\begin{itemize}
\item periodic boundary conditions: $\u_0$, $\omega$, $\phi$, $\phi_x$, $\phi_y$ are $1$-periodic vector-functions
\item homogeneous Dirichlet boundary conditions: $\u_0=0$ on $\partial\Omega$ and $\phi=0$ on $\partial\Omega$
\item homogeneous Neumann boundary conditions: no constraints on $\phi$, $\nabla\omega\cdot\vec{n}=0$ on $\partial\Omega$
\end{itemize}
In what follows we will be working with the case of periodic boundary conditions. By using the same argument as in~\cite[p.254]{Temam2001} it is not hard to prove that in this case there exists the unique solution $\omega\in C(0,T,H^1(\Omega))$ of~\eqref{eq:vorticity}, provided the initial condition is from $H^1(\Omega)$ and $f\in L^2(0,T,\Lt)$.

Analogously to~\cite{ZhukTTCDC15} we introduce a finite dimensional FG model for~\eqref{eq:vorticity}. We recall that $\{\phi_{cd}\}_{c,d\in \mathbb{Z}}$ is a total orthonormal system in $\Lt$:  $(\phi_{cd},\phi_{pq})_{\Lt}=\delta_{cp}\delta_{dq}$. In what follows we will use a simplified notation for double indicies, for instance $\{\phi_{cd}\}_{N}$ will refer to the vector $\{\phi_{cd}\}_{|c|\le\frac{N_1}2,|d|\le\frac{N_2}2}$, and $\{(\mathcal D\phi_{cd},\phi_{pq})_\Lt\}_{N,N}$ will refer to the matrix $\{(\mathcal D\phi_{cd},\phi_{pq})_\Lt\}_{|c|,|p|\le\frac{N_1}2,|d|,|q|\le\frac{N_2}2}$, where $N:=(N_1+1)(N_2+1)$. Define a linear $N$-dimensional subspace $L^N:=\operatorname{Lin}\{\phi_{cd}\}_{N}\subset\Lt$, and set \begin{equation}
\omega_N(x,y,t):=\sum_{|s_i|\le \frac{N_i}2} \omega_{s_1,s_2}(t) \phi_{s_1,s_2}(x,y)\,,\quad \omega_{cd}(t):=(\omega(t),\phi_{cd})_{\Lt}
\end{equation} with $\omega_{0,0}:=0$ (so that $\omega_N$ has zero mean). Clearly, $\omega_N$ is the projection of $\omega$ onto $L^N$. To approximate the projection coefficients $\omega_{cd}$ we restrict~\eqref{eq:vorticity} to $L^N$, i.e. we let $\phi$ run through $\{\phi_{cd}\}_{N}$, and substitute $\omega$ with $\omega_N$ in the resulting finite system of differential equations. We get the following FG model: \begin{equation}
%\dot\omega_{s_1,s_2} + \sum_{|k_i|\le \frac N2} (b(\u,\phi_k,\phi_s)+(\nabla\phi_k,\nabla\phi_s)_{\Lt})\omega_k(t) =(\mathcal{D}f,\phi_s)_{\Lt}\,.
\begin{split}
&\dot\omega_{cd} + b(\bar u + \nabla^\perp\psi_N,\omega_N,\phi_{cd})+\nu a(\omega_N,\phi_{cd})
=(\mathcal D f, \phi_{cd})_{\Lt},\\
&a(\psi_N,\phi_{cd}) = (\omega_N,\phi_{cd})_{\Lt}\,.
  \end{split}
\end{equation}
By using the orthogonality of $\{\phi_{cd}\}_{c,d\in \mathbb{Z}}$, we arrive at the following ODE: \begin{equation}
\begin{split}
  \dot \omega_{cd}(t)=&- \sum_{p,q,n,m}\frac{\omega_{pq}\omega_{nm}(pm-qn)L_xL_y \delta_{p+n,c} \delta_{q+m,d}}{p^2 L_y^2+q^2 L_x^2}\\
& -  \omega_{cd} (\frac{2\pi ic \tilde u}{L_x} - \frac{2\pi id\tilde v}{L_y})
- \nu \omega_{cd}(t) \bigl(\frac{4\pi^2 c^2}{L_x^2} + \frac{4\pi^2 d^2}{L_y^2}\bigr)\\
& + \sum_{n,m}(\mathcal D \phi_{mn}, \phi_{cd})_{\Lt}f_{mn}
\end{split}
\end{equation}
or, in the vector form,
\begin{equation}
  \label{eq:FGmodel}
\dfrac{d\w}{dt} = B(\w)\w + B(\u) \w + A\w + D\f\,, \w(0)=\w_0\,,
\end{equation}
where $\w:=\{\omega_{cd}\}_{N}$ is the vector of projection coefficients representing $\omega$ in $L^N$, $\w_0$, $\f$ represent $\omega(\cdot,\cdot,0)$ and $f$ in $L^N$, $A:=-\nu\operatorname{diag}(\lambda_{-\frac{N_1}2,-\frac{N_1}2}\dots \lambda_{\frac{N_1}2,\frac{N_1}2})$ represents the Laplacian $\Delta$ in its eigen-subspace $L^N$, and \begin{equation}
B(\w) = \bigl\{-\sum_{p,q}\frac{\omega_{pq}(pm-qn)L_xL_y}{p^2 L_y^2+q^2 L_x^2}\delta_{p+n,c} \delta_{q+m,d}\bigr\}_{N,N}\,,
\end{equation} represents the convection operator induced by the trilinear form $b$ (see \cite[p.279]{Temam2001}) in $L^N$. Finally, $D:=\{(\mathcal{D}\phi_{cd},\phi_{pq})_{\Lt}\}_{N,N}$.

% \textbf{Complex conjugacy of $\w$.} Note that $\w = \Psi \Psi^\star\w$ provided $\Psi$ is a $(N_1+1)(N_2+1)$ projection matrix defined by: $\Psi=\frac1{\sqrt{2}}\left(
%   \begin{smallmatrix}
%     1&0&\hdots&0&0&0&\hdots&0&1\\
%     0&1&\hdots&0&0&0&\hdots&1&0\\
%     \vdots&\vdots&\vdots&\vdots&\vdots&\vdots&\vdots&\vdots&\vdots\\
%     0&0&\hdots&1&0&1&\hdots&0&0\\
%     0&0&\hdots&0&0&0&\hdots&0&0\\
%     0&0&\hdots&1&0&-1&\hdots&0&0\\
%     \vdots&\vdots&\vdots&\vdots&\vdots&\vdots&\vdots&\vdots&\vdots\\
%     0&1&\hdots&0&0&0&\hdots&-1&0\\
%     1&0&\hdots&0&0&0&\hdots&0&-1
%   \end{smallmatrix}
% \right)$, i.e. $\w_{-\frac{N_1}2+k,-\frac{N_2}2+s}=\overline{\w}_{\frac{N_1}2-k,\frac{N_2}2-s}$ for any $0\le k,s\le \frac{N_1}2,\frac{N_2}2$ and $\w_{0,0}=0$.

\textbf{Complex conjugacy of $\w$.} Note that $\w = \Psi \Psi^\star\w$ provided $\Psi$ is a $(N_1+1)(N_2+1)$ projection matrix defined by: $\Psi=\frac1{\sqrt{2}}\left(
  \begin{smallmatrix}
    I_n&0_{n,1}&J_n\\
    0_{1,n}&0_{1,1}&0_{1,n}\\
    J_n&0_{n,1}&-I_{n}
  \end{smallmatrix}
\right)$, where $I_n$ and $J_n$ are respectively the identity and row-reversed identity matrices of size $n=(N_1+1)(N2+1)$, and $0_{a,b}$ represent zero matrices of size $a \times b$, i.e. $\w_{-\frac{N_1}2+k,-\frac{N_2}2+s}=\overline{\w}_{\frac{N_1}2-k,\frac{N_2}2-s}$ for any $0\le k,s\le \frac{N_1}2,\frac{N_2}2$ and $\w_{0,0}=0$.

\textbf{Skew-symmetry of the bilinear term.} Assume that $\div(\u)=0$ and $\u,w,v$ are smooth $1$-periodic functions on $\Omega$. We find integrating by parts that the trilinear form $b$ is skew-symmetric:
\begin{equation}
  \label{eq:trilin_form_wort}
  b(\u,w,v)=-b(\u,v,w)\,.
\end{equation}
Hence, the convection operator induced by $b$ is skew-symmetric too, and, as a result, the $\Lt$-norm of the vorticity, the enstrophy is not increasing, provided $D=0$, and is conserved if, in addition, $A=0$. This implies that (i) $B(\w)=-B^\star(\w)$ as $b(\u,\phi_{cd},\phi_{pq})=-b(\u,\phi_{pq},\phi_{cd})$ by~\eqref{eq:trilin_form_wort} so that $B(\vec w)$ is a skew-symmetric matrix, and (ii) $B(\vec u)$ is a diagonal matrix. By using the skew-symmetry of $B$ it is not hard to prove the unique solvability for~\eqref{eq:FGmodel} from any initial condition and for any $L^\infty$-input $\f$. Indeed, it is sufficient to take the inner product of both sides of~\eqref{eq:FGmodel} with the complex conjugate of $\w$, bound $(D\f,\w)_{C^N}$ by Schwartz inequality, recall that $(Q\f,\f)_{C^N}<1$ and use Bellman lemma to get a bound on the norm of $\w$.

\section{Problem statement}
\label{sec:problem-statement}
Assume that $\u=0$ and let $\w$ solve
\begin{equation}
  \label{eq:state}
  \dfrac{d\w}{dt} = B(\w)\w + A\w + D\f\,,\quad \w(0)=\w_0\,,
\end{equation}
and assume that a vector-function $\y$ is observed in the following form:
\begin{equation}
  \label{eq:obs}
\y(t) = H\w(t) + F\e(t)\,,
\end{equation}
where $H = \{(H_{cd},\phi_{s_1,s_2})_{\Lt}\}_{M,N}$, $H_{cd}$ is an averaging kernel (e.g. a smooth function with compact support in a vicinity of a grid point $x_{cd}$) and $\e=\{\eta_{cd}\}_{M}$ is a measurable vector-function modelling noise in the output, $F$ is a given matrix.

We further assume that the tuple $(\w_0,\f,\e)$ is an uncertain element of the following $L^\infty$-type ellipsoid:
\begin{equation}
\begin{split}
\E:=\{\w_0:(S^{-1} \w_0)\cdot\w_0\le 1\}\times
\{(\f,\e): Q^{-1}(t)\f(t)\cdot\f(t) + R^{-1}(t) \e(t)\cdot \e(t) \le 1\}\,.
\end{split}
\end{equation}
where $Q$, $R$ and $S$ are given positive definite matrices of appropriate dimensions. Given $(\w_0,\f,\e)\in\E$, $\w(\cdot;\w_0,\f)$ refers to the unique solution of~\eqref{eq:state}, which corresponds to $\w_0$ and $\f$, and $\y(\cdot;\e)$ refers to $\y$ which corresponds to $\w(\cdot;\w_0,\f)$ and $\e$ through~\eqref{eq:obs}.

We say that $\hw$ is an estimate of $\w$ in the form of a filter if $\hw$ solves the following equation:
\begin{equation}
  \label{eq:filter}
  \dfrac{d\hw}{dt} = B(\hw)\hw + A\hw + PH^\top(\y-H\hw)\,, \quad \hw(0)=0\,,
\end{equation}
for a symmetric matrix-valued function $P(t)$, the gain. We will write $\hw(\cdot;\y,P)$ to stress the dependence of $\hw$ on $\y$ and $P$. Define the estimation error $e:=\w(t;\w_0,\f)-\hw(t;\y(\cdot;\e),P)$ and set:
\begin{equation}
  \label{eq:error}
\sigma(t;\w_0,\f,\e,\y,P):=(e,e)_{C^N}\,.
\end{equation}
In what follows, most of the time we will be using a simplified notation, e.g., $\sigma(t)$ or $\sigma$ instead of $\sigma(t;\w_0,\f,\e,\y,P)$, or $\hw$ instead of $\hw(\cdot;\y,P)$. %In some cases, however, the ``complete'' notation will be used to stress the dependence of $e$ or $\hat x$ on some of the ``hidden'' parameters.

\textbf{Our goal is}, given $\varepsilon>0$, to find a symmetric $P(t)$ such that
\begin{equation}
  \label{eq:goal}
\max_{(\w_0,\f,\e)\in \E} \sigma(t;\w_0,\f,\e,\y,P) \le  \varepsilon\,,\forall t>t^*>0\,.
\end{equation}

\section{Main results}
\label{sec:main-results}

In this section we present sufficient conditions for~\eqref{eq:goal} to hold, namely an algebraic Riccati inequality with time-dependent matrix coefficients $P$ which ensures the exponential decay of the estimation error $\sigma$ for the generic $L^\infty$-type uncertainty description. As a conjecture, we suggest that $\sigma$ coincides with a solution of a HJB equation along the trajectories of~\eqref{eq:state}. Next, we propose a computationally feasible version of the aforementioned sufficient conditions, namely a linear matrix inequality for $P$, which enforces~\eqref{eq:goal}, provided $F=0$.
\begin{theorem}[$L^\infty$-type uncertainty]\label{t:Linfty}
Let $q>0$ and define $B_1(\hw) = \bigl(\begin{smallmatrix} B(\Psi e_1) \hw\dots B(\Psi e_{(N_1+1)(N_2+1)}) \w
\end{smallmatrix}\bigr)\Psi^\star$. If $P$ verifies the following matrix inequality:
\begin{equation}
  \label{eq:ARE}
\begin{split}
  B_1(\hw) + B_1^\star(\hw) + DQ D^\star - PH^\star H - H^\star HP + PH^\star F R F^\star H P < -q I\,,
\end{split}
\end{equation}
then
\begin{equation}
\begin{split}
  \label{eq:error_est}
  \max_{(\w_0,\f,\e)\in \E} \sigma(t;\w_0,\f,\e,\y,P) \le C_1(q)+C_2(q) e^{-(2|\lambda(A)|+q)t}\,,
\end{split}
\end{equation}
where $C_1(q):=\frac{1}{2|\lambda(A)|+q}$, $C_2(q):=\lambda(S)-\frac{1}{2|\lambda(A)|+q}$, $\lambda(X)$ denotes the maximal eigenvalue of the matrix $X$.
\end{theorem}
\begin{proof}
Take any $(\w_0,\f,\e)\in\E$ and let $\w(\cdot;\w_0,\f)$ be the corresponding unique solution of~\eqref{eq:state}, and  $\y(\cdot;\e)$ be the corresponding output. Suppose that $\hw(\cdot;\y,P)$ solves~\eqref{eq:filter} for this particular $\y(\cdot;\e)$ and a symmetric $P$. Recall~\eqref{eq:state}-\eqref{eq:filter}. We find that:
\begin{align*}
\dot e &= %B(\w)\w - B(\hw)\hw + Ae - PH^\star H e +D\f- PH^\star F\e \\
%&= B(\w)\w - B(\w)\hw + B(\w)\hw - B(\hw)\hw + A e + D\f- PH^\star H e - PH^\star F\e \\
%&= B(\w)e + B(e)\hw + Ae - PH^\star H e + D\f - PH^\star F\e \\
B(\w)e + B_1(\hw)e + Ae - PH^\star H e  + D\f - PH^\star F\e\,.
\end{align*}
Hence \begin{equation}
\begin{split}
  (e,\dot e)_{C^N}&=\frac12\dfrac{d}{dt}(e,e)_{C^N}\\
& = (e, (B(\w) + B_1(\hw) + A  - PH^\star H) e)_{C^N}\\
  & + (e, D\f - PH^\star F\e)_{C^N}\,.
\end{split}
\end{equation} Since $B(\w)=-B^\star(\w)$ it follows that $(e, B(\w)e)_{C^N}=0$, and we obtain the following equation:
\begin{equation}
  \label{eq:error_ode}
  \begin{split}
    \dot\sigma(t) &= (e, (B_1(\hw) + B_1^\star(\hw) - PH^\star H - H^\star HP)e)_{C^N}\\
    &+ 2(Ae,e)_{C^N}+ (2(e, Df - PH^\star Fg)_{C^N}\,.
  \end{split}
\end{equation}
It is easy to find by using the Schwarz inequality and the definition of $\E$ that:
\begin{equation}
  \label{eq:eDf}
\begin{split}
  2&(e, Df - PH^\star Fg)_{C^N}\\
&\le2(\left( D Q D^\star + PH^\star F R F^\star H P\right)e,e)_{C^N}^\frac12\\
&\le (DQ D^\star e,e)_{C^N} + (PH^\star F R F^\star H P e,e)_{C^N} +1\,.
\end{split}
\end{equation}
Assume now that $P$ solves \eqref{eq:ARE} and recall that $A\le 0$. It then follows from~\eqref{eq:error_ode}-\eqref{eq:eDf} that \begin{equation}
\dot\sigma < 1 + (e, (2A - qI)e)\le 1 - (2|\lambda(A)|+q) \sigma\,.
\end{equation} Define $v(t):=\sigma(t)-\frac{1}{2|\lambda(A)|+q}$. We have that $\dot v< -(2|\Lambda(A)|+q)v(t)$ and so, by Bellman lemma, $v(t)\le v(0)e^{-(2|\alpha|+q) t}$. Hence, it follows that \begin{equation}
\sigma(t;x_0,f,g,y,P)\le C_1+e^{-(2|\alpha|+q) t} (\sigma(0) -\frac{1}{2|\lambda(A)|+q})\,.
\end{equation} Combining this with that $\sigma(0) = (\w(0),\w(0))$, and $(S^{-1}x(0),x(0))_{C^N}\le 1$, and by noting that \begin{equation}
\max_{z:(S^{-1}z,z)_{C^N}\le 1}(z,z) = \max_{l:(l,l)=1} (Sl,l)^2 = \lambda(S)\,
\end{equation} we obtain~\eqref{eq:error_est}.
\end{proof}
It is not hard to see that for any $\varepsilon>0$ one can find $q>0$ such that \begin{equation}
C_1(q) + C_2(q)e^{-(2|\lambda(A)|+q)}<\varepsilon\,.
\end{equation} The structure of the equation~\eqref{eq:error_ode} suggests the following conjecture:
\begin{Conjecture}
Let \begin{equation}
\begin{split}
  X_T:= \{z: &\exists (\w_0,\f,\e)\in \E\text{ and } 0\le t^*\le T\text{ such that:}\\
  z &= \w(t^*)\\
   &\text{ and } \\
   &\dfrac{d\w}{dt} = B(\w)\w + A\w + D\f\,, \w(0)=\w_0\}\,.
\end{split}
\end{equation}
Assume that $V$ solves the following HJB equation:
\begin{equation}\label{eq:hjb}
\begin{split}
&\partial_t V = \frac14(\partial_x V, (A +A^\star) \partial_x V)\\
&+\frac14(\partial_x V,\left( B_1(\hw) + B_1^\star(\hw) - PH^\star H - H^\star HP \right)\partial_x V)\\
&+ \max_{(\f,\e)\in\E}(\partial_x V, D\f - PH^\star F\e)\\
&V(x,0) = (x,x)\,, \quad V(x,t) = (x,x) \text{ on } \partial X_T\,,
\end{split}
\end{equation}
Then \begin{equation}
\max_{\w_0,\f,\e}\sigma(t;\w_0,\f,\e,\y,P) = V(\w_*(t),t)\,.
\end{equation} where $\w_*$ corresponds to $\w_0$ and $\f$ at which the $\max$ above is attained.
\end{Conjecture}
Solvability conditions and numerical methods for~\eqref{eq:ARE} are known~\cite{care}. However, computing the numerical solution of \eqref{eq:ARE} in high dimensions is a very challenging problem. We stress that~\eqref{eq:ARE} simplifies to a LMI provided $F=0$:
\begin{corollary}[Exact output]
Let $F=0$, $q>0$ and assume that $P$ solves the following LMI $$
B_1(\hw) + B_1^\star(\hw) + DQ D^\star - PH^\star H - H^\star HP < -q I\,.
\eqno(LMI)
$$
Then~\eqref{eq:error_est} holds true.
\end{corollary}
\begin{proof}
  This is a straightforward consequence of the Theorem~\ref{t:Linfty}.
\end{proof}
\subsection{Computational form of the filter}
\label{sec:discretization}
The most straightforward approach of solving $(LMI)$ is to solve the following linear Lyapunov equation:
$PH^\star H + H^\star HP = W(t):=B_1(\hw(t)) + B_1^\star(\hw(t)) + DQ D^\star + qI$. The latter may not have the classical solution as its right hand side $W(t)$ may not belong to the range of the linear operator $P\mapsto V(P):=PH^\star H + H^\star HP$. On the other hand, one can always compute the least-squares solution of the linear equation $V(P)=W$. Indeed, this amounts to evaluating $\hat P(t):=(I\otimes H^\star H + H^\star H\otimes I)^+\operatorname{vec}(W(t))$, where $Q^+$ denotes the pseudoinverse of $Q$, and $\operatorname{vec}(P)$ is the vector formed by stacking the columns of $P$ one upon another. Note that the matrix $B_1(\hw) + B_1^\star(\hw)+qI$ is usually very sparse so that $\hat P_t$ can be effectively computed by using a standard least-squares solver (e.g. \textrm{GMRES}). For sparse matrices this approach appears to be more efficient than applying the standard SDP solvers to solve $(LMI)$ directly, especially in high dimensions.

We stress that $\hat P$ solves $(LMI)$ provided the real spectrum of the residual $W (t)-V(\hat P)$ belongs to $(-\infty,0)$ which gives us the pointvise \emph{detectability conditions}, i.e. the eigen-values of the projection of $W(t)$ onto the orthogonal completion of the range of $P\mapsto V(P)$, the ``unobservable'' eigenvalues of $W$, must be negative. On the other hand, $B_1$ is a linear function of $\hw$ and so does $\hat P$ defined . As a result, the filtering equation has a bilinear correction term $\hat P(\hw)H^\star H\hw$. To solve~\eqref{eq:filter} numerically one can use a modification of the algorithm proposed in~\cite{ZhukTTCDC15}: namely, define $J(\hw_t):=B(\hw_t) + A - \hat P_t H^\top H$, set $F_{t,t+t}:=\hat P_tH^\top\frac{\y_{t+1}+\y_t}{2}$ and compute $\hw_{t+1}$ given $\hw_t$ as follows: $\hw_0=0$ and
\begin{align}
&\frac{\hw_{t+1}-\hw_t}{dt} = J(\hw_t)\frac{\hw_{t+1}+\hw_{t}}2 + F_{t,t+1}\,,   \label{eq:filter_dt}\\
&\hat P_t = (I\otimes H^\star H + H^\star H\otimes I)^+\operatorname{vec}(W(t)) \label{eq:gain_dt}
\end{align}

\subsection{Numerical experiment}
\label{sec:num-example}

\textbf{Synthetic observations.} To generate observations we set $N_1=N_2=40$, $N=(N_1+1)(N_2+1)$ and compute the numerical solution of~\eqref{eq:state}, the ``true'' vorticity by using the numerical algorithm~\eqref{eq:filter_dt} with  $J(\hw_t):=B(\hw_t) + A$, $F_{t,t+1}=\frac{\f_{t+1}+\f_t}2$ and $\hw(0)=\w_0$, where $\w_0$ is the projection of $\omega(x,y,0)= 3(1-(x-\pi))^2e^{-(x-pi)^2- (y-\pi+1)^2}- 10(5^{-1}(x-\pi) - (x-\pi)^3 - (y-\pi)^5)e^{-(x-\pi)^2-(y-\pi)^2}- 3^{-1}e^{-(x-\pi+1)^2 - (y-\pi)^2}$ onto $L^N$. The timestep is taken to be $dt=0.025$ and $T=20$. The forcing $f$ is taken to be proportional to the sum of two basis functions, so that the vector matrix $D$ is a diagonal matrix with zero entries but two at positions $\frac{(N_1+1)N_1}{2}\pm |d|+\frac{N_2}2+1$, $d=6$. $\f$ is taken to be a constant vector with components equal to $d/2$. The diffusion coefficient is set to $\nu=0.005$ indicating a weak damping effect. %Figure~\ref{fig:1} compares the heat map of the forced vorticity ($D\ne 0$) versus unforced vorticity ($D=0$) over the $200x200$ evenly spaced grid. Clearly, the difference between the forced ($D=I$) and unforced dynamics ($D:=0$) is significant even though both simulations start from the same initial condition displayed at Figure~\ref{fig:init}. In the literature on geophysical fluids the behaviour shown on Figures~\ref{fig:1}-\ref{fig:estimates} is often referred to as the pseudo-turbulence.

\textbf{Estimation.} The filter is computed according to~\eqref{eq:filter_dt}-\eqref{eq:gain_dt}. $H$ is taken to be an identity matrix with only non-zero entries representing the following mode numbers: $(-6,-3:3,6)\times (-6,-3:3,6)$, $-3:3$ stands for $\{-3,-2,-1,0,1,2,3\}$. Hence, we observe just $81$ components of the $1681$-dimensional state vector $\w$. The initial condition for the filter is set to $0$ and the forcing $\f$, used to generate observations, is assumed to be unknown, $R:=I$, $F=0$ and $Q=2\|\f\|^{-1}_{C^N}I$. We set $q:=200\max(Q)$. The observed modes are subject to a small (upd to $10\%$ signal to noise ratio) random noise drawn from the uniform distribution over $(-.2/\sqrt{N},.2\sqrt{N})$. Figure~\ref{fig:estimates} shows the estimate and truth at different times. Figure~\ref{fig:rel_err} displays the dynamics of the relative estimation error over time. As noted, $(LMI)$ does not hold true, yet the error converges to $0$.
\begin{figure}
        \centering
        \begin{subfigure}[b]{0.5\textwidth}
                \includegraphics[width=\textwidth]{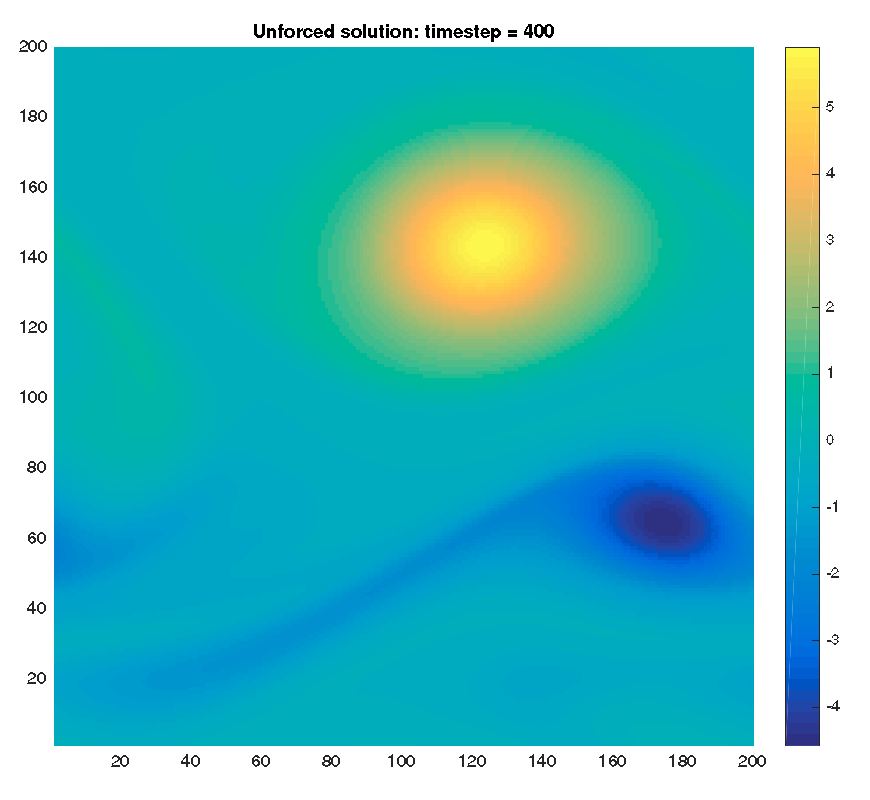}
                \caption{Unforced solution, $t=10$}
                \label{fig:estimate2}
        \end{subfigure}%
        ~ %add desired spacing between images, e. g. ~, \quad, \qquad, \hfill etc.
        \begin{subfigure}[b]{0.5\textwidth}
                \includegraphics[width=\textwidth]{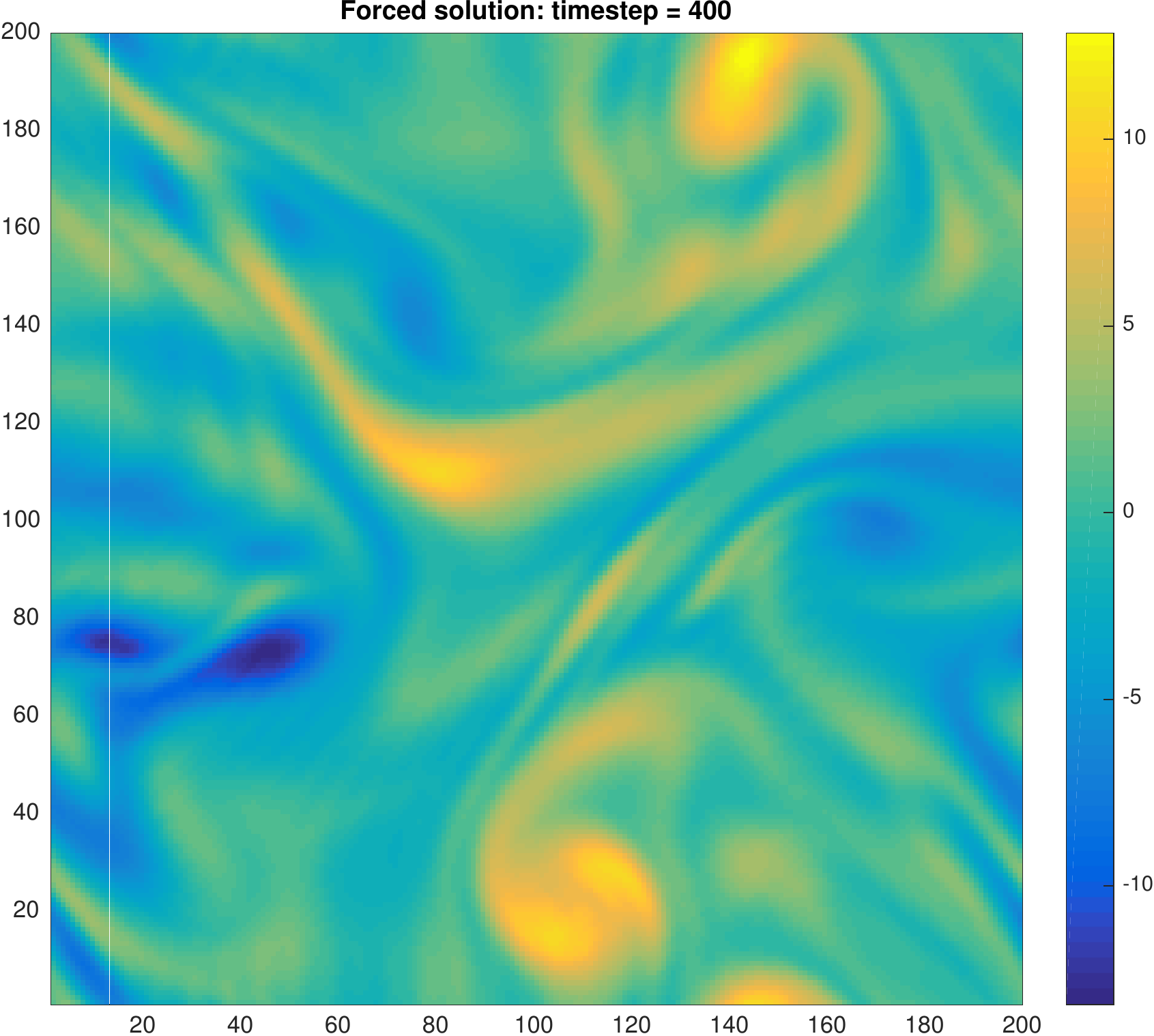}
                \caption{Forced solution, $t=10$}
                \label{fig:truth2}
        \end{subfigure}%
	\caption{The impact of forcing}\label{fig:1}
\end{figure}
\begin{figure}
        \centering
        \begin{subfigure}[b]{0.5\textwidth}
                \includegraphics[width=\textwidth]{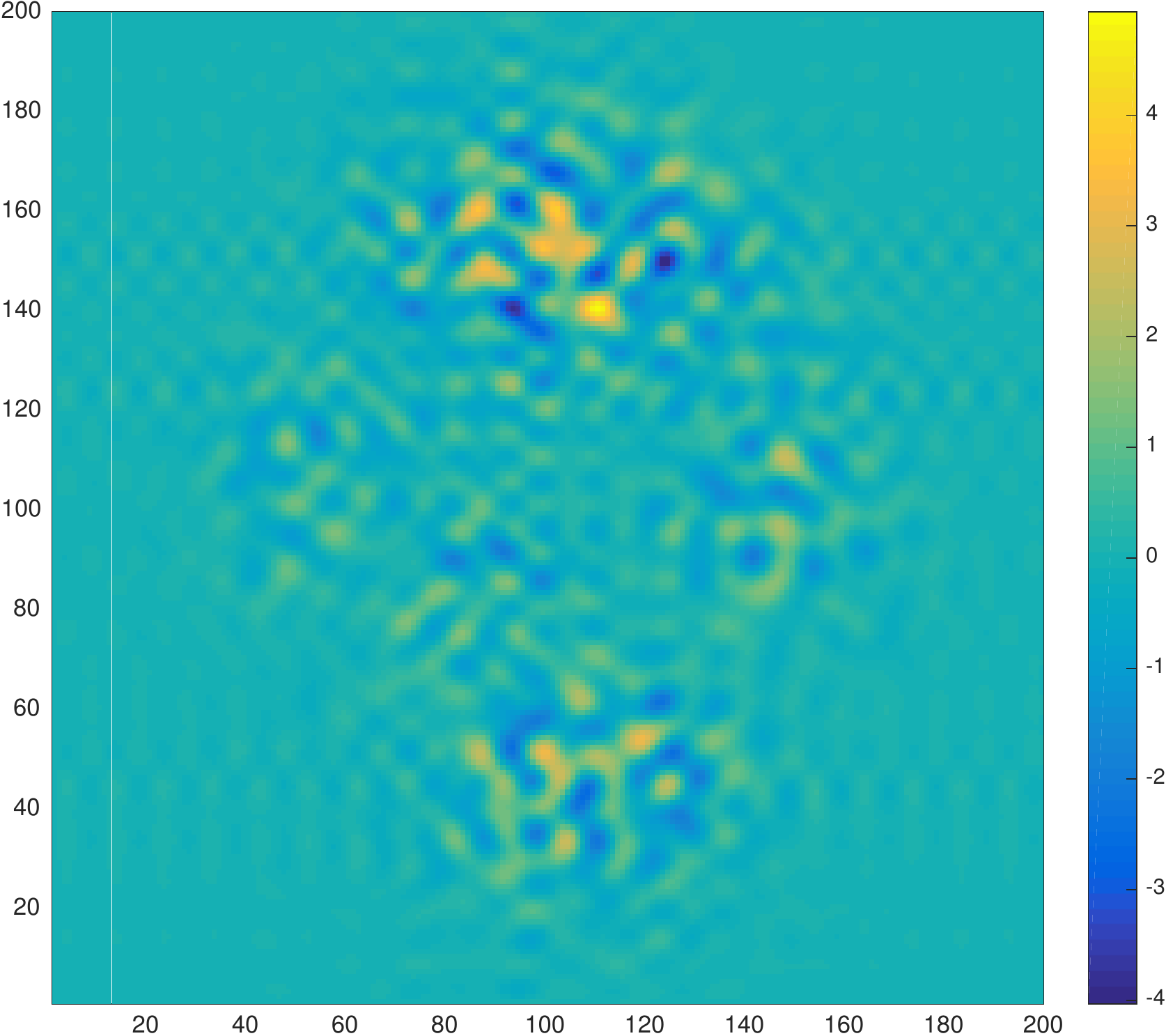}
                \caption{The ``true'' vorticity at time $t=0$}
                \label{fig:init}
        \end{subfigure}%
        ~ %add desired spacing between images, e. g. ~, \quad, \qquad, \hfill etc.
        \begin{subfigure}[b]{0.5\textwidth}
                \includegraphics[width=\textwidth]{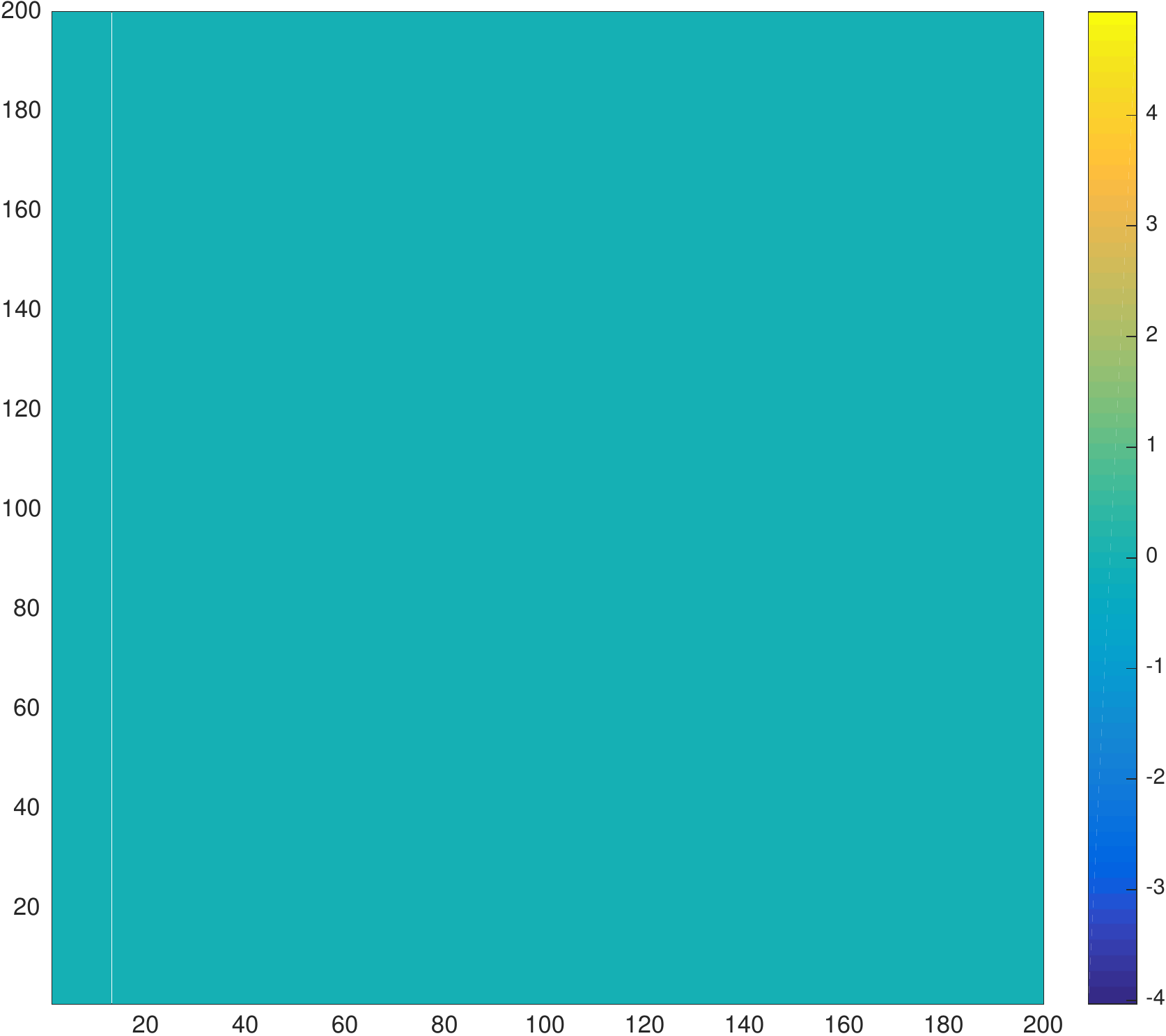}
                \caption{The estimate at time $t=0$, rel.err. = $100\%$}
                \label{fig:truth1}
        \end{subfigure}%
        ~ %add desired spacing between images, e. g. ~, \quad, \qquad, \hfill etc.

        \begin{subfigure}[b]{0.5\textwidth}
                \includegraphics[width=\textwidth]{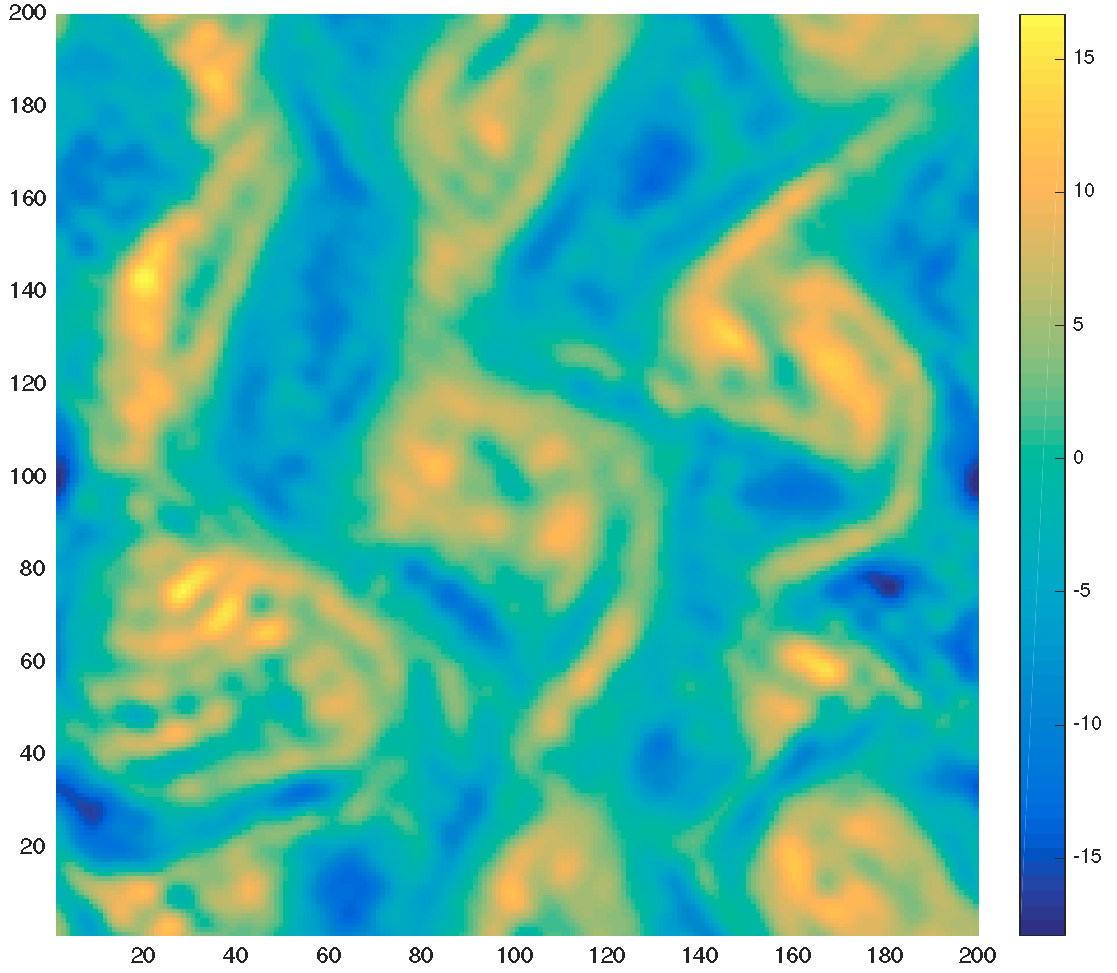}
                \caption{The ``true'' vorticity at time $t=5$}
                \label{fig:estimate2}
        \end{subfigure}%
        ~ %add desired spacing between images, e. g. ~, \quad, \qquad, \hfill etc.
        \begin{subfigure}[b]{0.5\textwidth}
                \includegraphics[width=\textwidth]{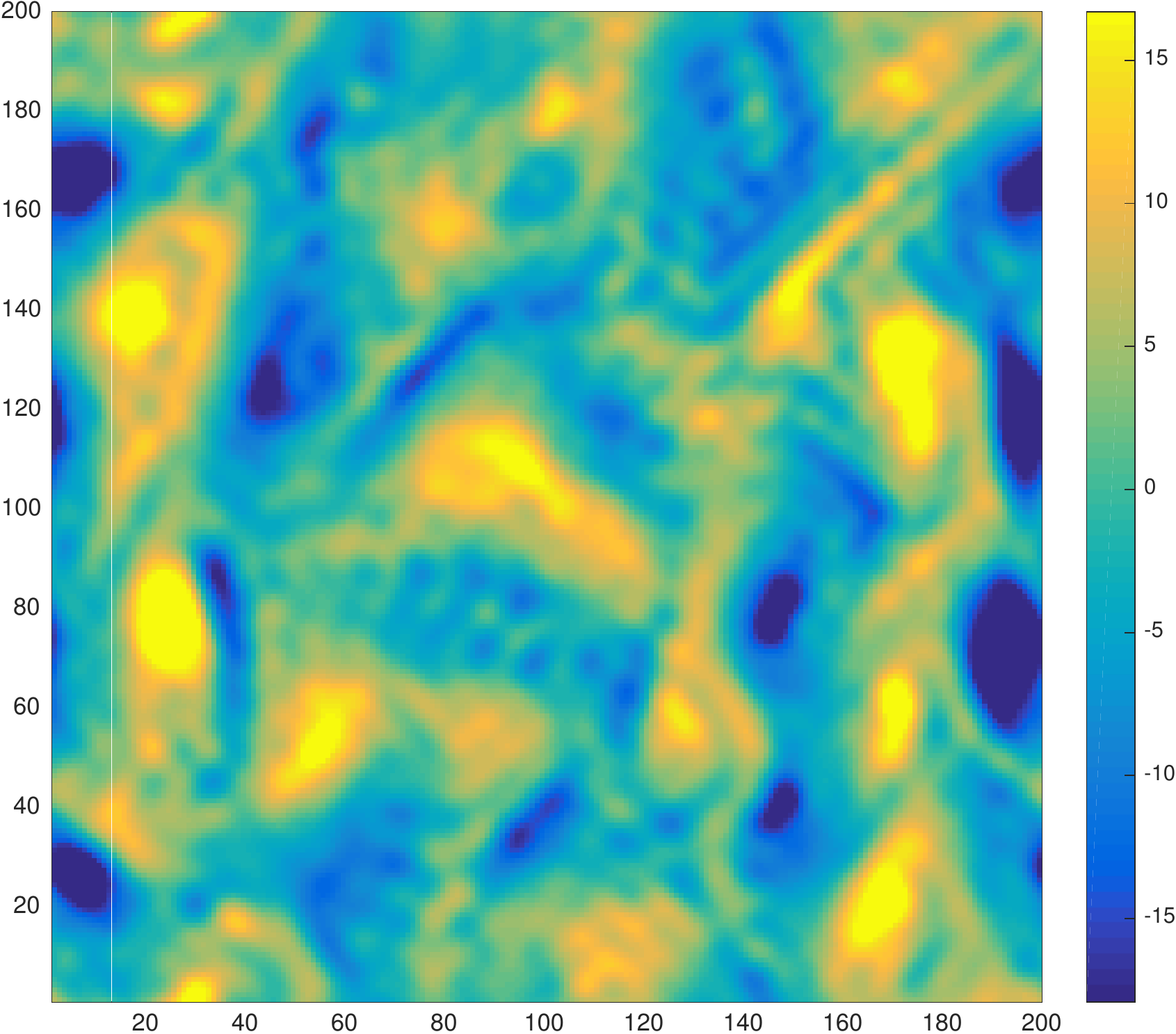}
                \caption{The estimate at $t=5$, rel.err. = $12\%$}
                \label{fig:truth2}
        \end{subfigure}%

        \begin{subfigure}[b]{0.5\textwidth}
                \includegraphics[width=\textwidth]{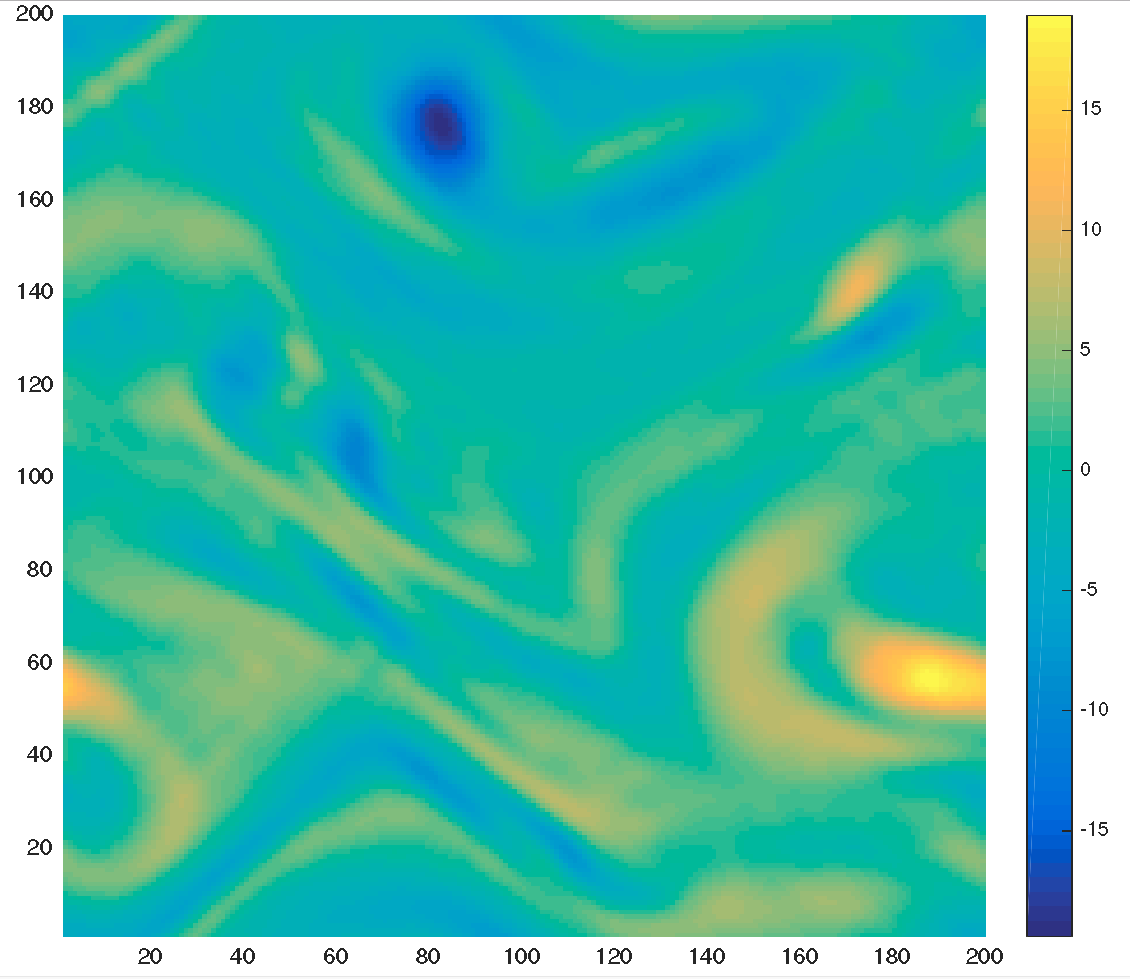}
                \caption{The ``true'' vorticity at time $t=15$}
                \label{fig:estimate3}
        \end{subfigure}%
        ~ %add desired spacing between images, e. g. ~, \quad, \qquad, \hfill etc.
        \begin{subfigure}[b]{0.5\textwidth}
                \includegraphics[width=\textwidth]{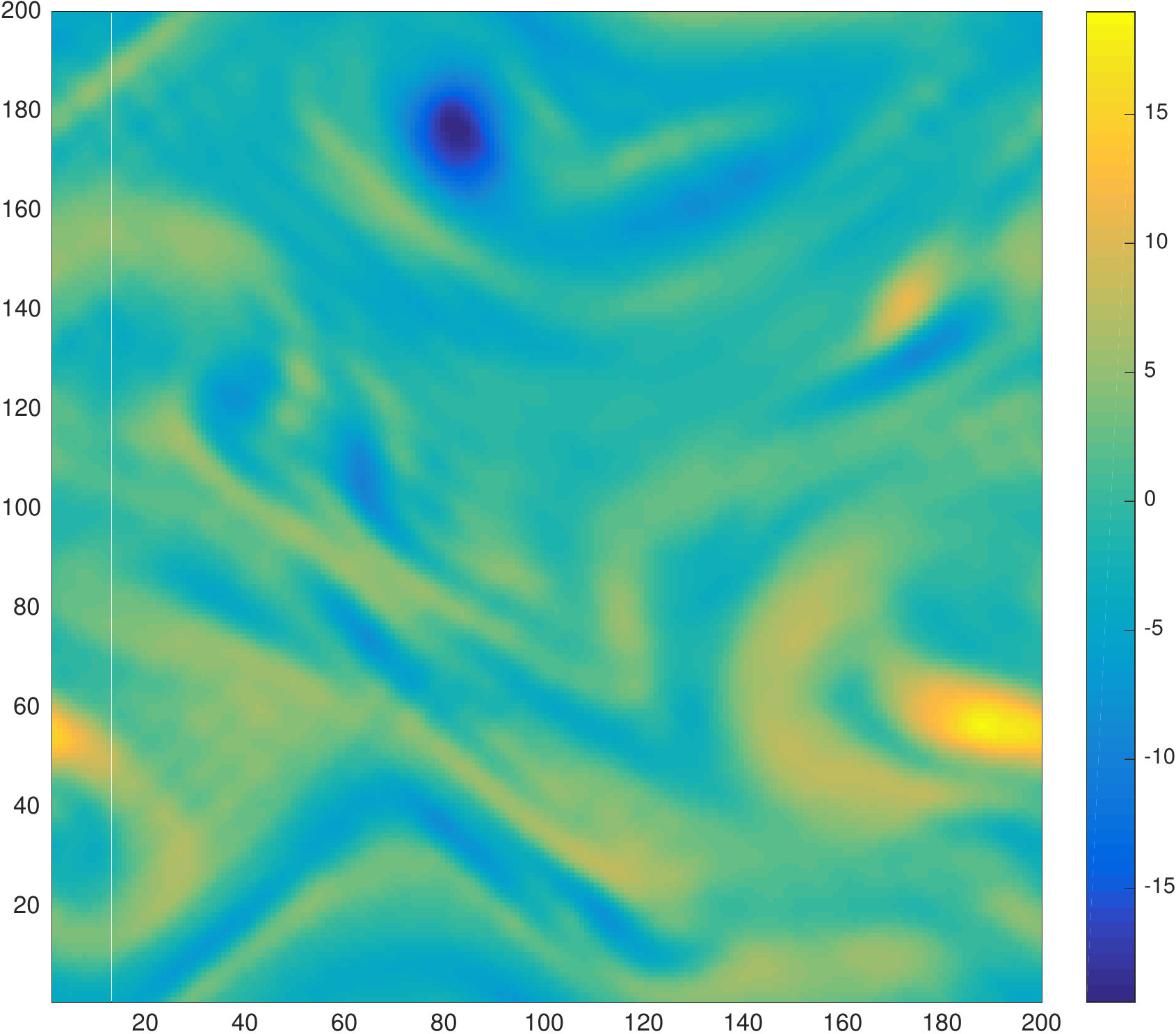}
                \caption {The estimate at $t=15$, rel.err. = $12\%$}
                \label{fig:truth3}
        \end{subfigure}%
	\caption{The ``true'' vorticity and the estimates for different times}\label{fig:estimates}
\end{figure}
\begin{figure}\centering
                \includegraphics[width=0.8\textwidth]{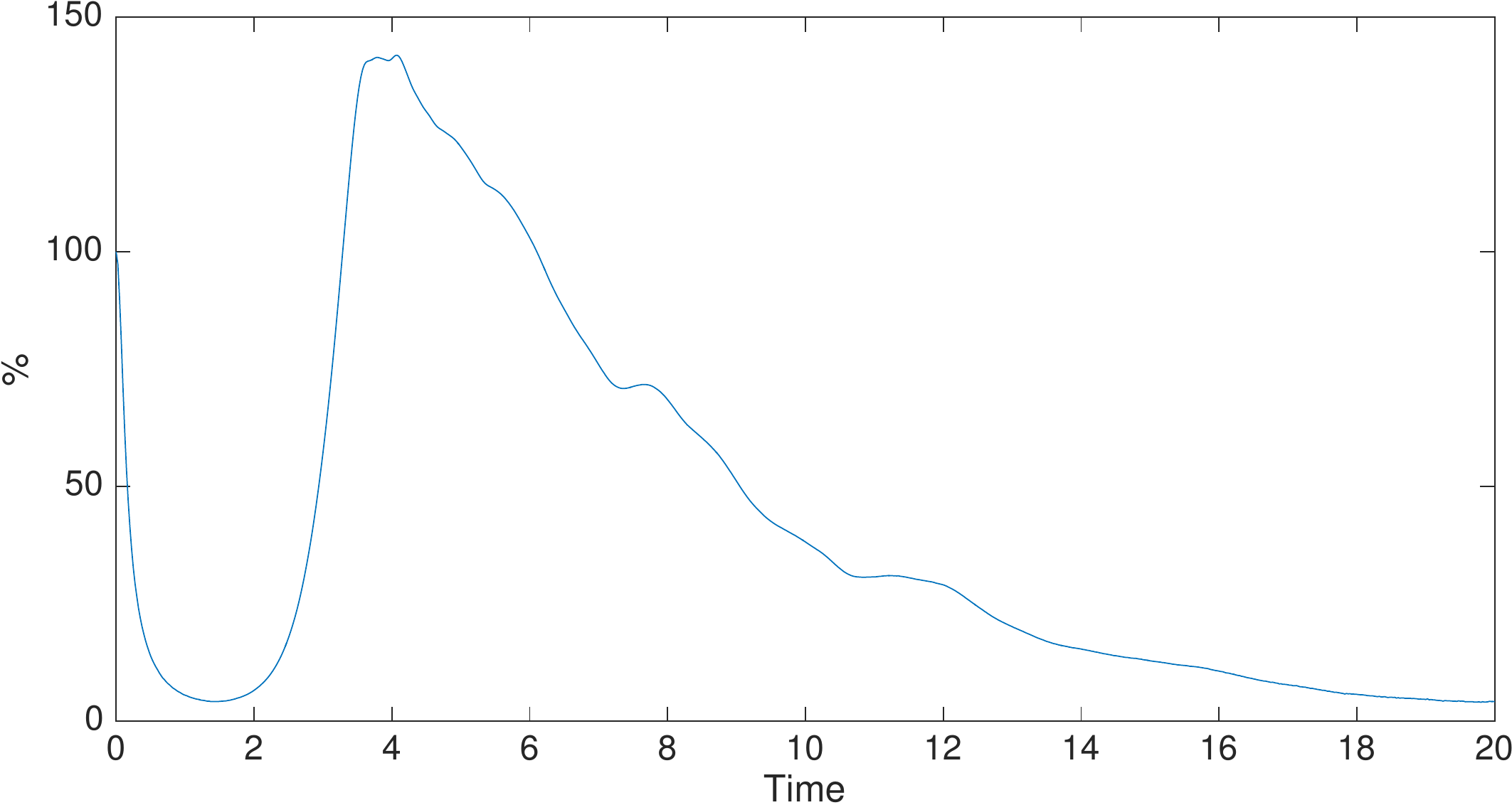}
                \caption{Relative $L^2(\Omega)$ estimation error over time}
                \label{fig:rel_err}
              \end{figure}
\section{Conclusion}
\label{sec:conclusion}
The paper presented a new data assimilation algorithm for Navier-Stokes equations which is based upon the skew-symmetry of the non-linear term. The algorithm can be applied to generic bilinear systems or skew-symmetric nonlinear systems without major revisions. A very challenging topic for the future research is to investigate the relation between the proposed sufficient conditions and positive/negative Lyapunov exponents of a bilinear equation, and to relax the pointvise LMI to a condition including a long term averages.

%\bibliography{refs,myrefs}
%\bibliography{refs,myrefs}
\end{document}